\documentclass[10pt]{amsart}
\usepackage{enumerate,amsmath,amssymb,latexsym,
amsfonts, amsthm, amscd, mathabx}

\theoremstyle{plain}

\newtheorem{theorem}{Theorem}

\numberwithin{equation}{section}

\newcommand{\PP}{\mathbf{P}}
\newcommand{\PPi}{\mathbf{\Pi}}

\newcommand{\h}{\hash}

\addtocounter{section}{-1}

\begin{document}

\title {Projective Space: Harmonicity and Projectivity}

\date{}

\author[P.L. Robinson]{P.L. Robinson}

\address{Department of Mathematics \\ University of Florida \\ Gainesville FL 32611  USA }

\email[]{paulr@ufl.edu}

 \keywords{}

\begin{abstract}

For an axiomatization of three-dimensional projective space based on points and planes, we discuss appropriate versions of the harmonicity axiom and the projectivity axiom, showing that each axiom is equivalent to its spatial dual. 

\end{abstract}

\maketitle

\medbreak

\section*{Introduction} 

In a recent sequence of papers on three-dimensional projective space, we considered a self-dual axiomatic framework that is based on lines, with points and planes as derived elements: in [2] we established the framework and showed that it is equivalent to the classical Veblen-Young framework [6] with its assumptions of alignment and extension; in [3] we presented a version of the axiom of projectivity that is appropriate to this framework; and in [5] we treated likewise the axiom of harmonicity. In a related paper [4] on three-dimensional projective space, we set up an equivalent self-dual axiomatic framework that is based instead on points and planes, with lines as derived elements. In the present paper, we round out the discussion of this point-plane framework by exploring appropriate versions of harmonicity and projectivity; we demonstrate that each of these versions is equivalent to its spatial dual, so that the inclusion of these axioms maintains the principle of duality. 

\medbreak 

\section*{The Self-dual Framework}

\medbreak 

For convenience, we recall briefly the framework established in [2]. The disjoint nonempty sets $\PP$ of {\it points} and $\PPi$ of {\it planes} are together equipped with a relation $\h$ of {\it incidence}. When $O \in \PP$ and $\omega \in \PPi$ we write $O^{\h} = \{ \pi \in \PPi : O \h \pi \}$ and $\omega^{\h} = \{ P \in \PP : \omega \h P \}$. More generally: if $\mathbf{S} \subseteq \PP$ then $\mathbf{S}^{\h} \subseteq \PPi$ comprises all planes that are incident to each point in $\mathbf{S}$; if $\mathbf{\Sigma} \subseteq \PPi$ then $\mathbf{\Sigma}^{\h} \subseteq \PP$ comprises all points that are incident to each plane in $\mathbf{\Sigma}$. Moreover, $\mathbf{S} \h \mathbf{\Sigma}$ is the statement that $P \h \pi$ whenever $P \in \mathbf{S}$ and $\pi \in \mathbf{\Sigma}$. In these terms, the axioms for our framework are four in number, as follows. 

\medbreak 

\noindent 
AXIOM [1] \\ 
If $P \in \PP$ then $P^{\h} \ne \PPi$. \\ 
If $\pi \in \PPi$ then $\pi^{\h} \ne \PP$. 

\medbreak 

\noindent 
AXIOM [2] \\ 
If $A, B \in \PP$ then $| \{ A, B \}^{\h} | > 2$. \\ 
If $\alpha, \beta \in \PPi$ then  $| \{ \alpha, \beta \}^{\h} | > 2.$

\medbreak 

\noindent 
AXIOM [3] \\
If $A, B, C \in \PP$ then $\{ A, B, C \}^{\h} \ne \emptyset.$ \\
If $\alpha, \beta, \gamma \in \PPi$ then $\{ \alpha, \beta, \gamma \}^{\h} \ne \emptyset.$

\medbreak 

\noindent 
AXIOM [4] \\
Let $A \ne B$ in $\PP$ and $\alpha \ne \beta$ in $\PPi$. \\
 If $\{ A, B \} \h \{\alpha, \beta \}$ then $\{ A, B \}^{\h} = \{ \alpha, \beta \}^{\h \h}$ and $\{ \alpha, \beta \}^{\h} = \{ A, B \}^{\h \h}.$

\medbreak 

These axioms support our defining a {\it line} as a subset of $\PP \cup \PPi$ having the form 
$$\ell = \ell_{\PP} \cup \ell_{\PPi}$$
with 
$$\ell_{\PP} = \{ A, B \}^{\h \h} = \{ \alpha, \beta \}^{\h} , \; \; 
\ell_{\PPi} = \{ \alpha, \beta \}^{\h \h} = \{ A, B \}^{\h}$$
for $A \ne B$ in $\PP$ and $\alpha \ne \beta$ in $\PPi$ such that $\{ A, B \} \h \{\alpha, \beta \}$; for convenience, we may write $\ell = A B$ or $\ell = \alpha \beta$ in this case. Points or planes are said to be {\it collinear} precisely when they are all elements of one line: thus, $\mathbf{S} \subseteq \PP$ is a set of collinear points iff $\mathbf{S} \subseteq \ell_{\PP}$ for some line $\ell$; also,  $\mathbf{\Sigma} \subseteq \PPi$ is a set of collinear planes iff $\mathbf{\Sigma} \subseteq \ell_{\PPi}$ for some line $\ell$. 

\medbreak 

It is perhaps worth drawing attention to an apparent peculiarity that arises from our defining a line as a special subset of $\PP \cup \PPi$. Let $\ell$ be a line, $P$ a point and $\pi$ a plane: the statement that $P$ lies in $\ell$ is symbolized naturally by $P \in \ell$; the statement that $\ell$ lies in $\pi$ is symbolized perversely by $\pi \in \ell$. 

\medbreak 

Theorem 1 in [4] asserts that if $A, B, C$ are not collinear then $\{A, B, C \}^{\h}$ is a singleton: in more traditional terms, three non-collinear points lie together in a unique plane; conversely, if $\{A, B, C \}^{\h}$ is a singleton then of course $A, B, C$ are not collinear. Theorem 2 in [4] shows that if $m$ and $n$ are distinct lines then $m_{\PP} \cap n_{\PP} \ne \emptyset \Leftrightarrow m_{\PPi} \cap n_{\PPi} \ne \emptyset$: further, that if these equivalent conditions are satisfied then $|m_{\PP} \cap n_{\PP} | = 1$ and $|m_{\PPi} \cap n_{\PPi} | = 1$; we may then write $m_{\PP} \cap n_{\PP} = \{ m \cdot n \}$ and $m_{\PPi} \cap n_{\PPi} = \{ m \square n \}$. In short: if two distinct lines share a plane, then they share a unique point; if two distinct lines share a point, then they share a unique plane. 

\medbreak 

Incidentally, Coxeter refers to his version of Axiom (A3) from [6] as `Veblen's ingenious device for declaring that any two {\it coplanar} lines have a common point {\it before} defining a plane'; see page 16 of [1]. In a similar but milder vein, our AXIOM [3] and AXIOM [4] together serve to guarantee that two coplanar lines have a common point before defining a {\it line}. 

\medbreak 

Finally, we note that the principle of {\it duality} holds in our point-plane framework: if point and plane are interchanged in the statement of any theorem, the result is a theorem; this is automatic, as our axioms themselves are unaltered by the interchange. The notions of {\it point} and {\it plane} are dual, while {\it line} is a self-dual notion. 

\bigbreak 

\section*{Harmonicity} 

\medbreak 

The axiom of {\it harmonicity} [5] is expressed in Section 18 of [6] as
\medbreak  
ASSUMPTION H$_0$: the diagonal points of a complete quadrangle are noncollinear 
\medbreak 
\noindent
and is elsewhere referred to as the {\it Fano} axiom. Here, a complete quadrangle is the planar figure comprising four points $O, P, Q, R$, no three of which are collinear, along with the three diagonal points $A = OP \cdot QR, \; B = OQ \cdot RP, \; C = OR \cdot PQ$. This axiom is already in a form suitable for importation into the present point-plane framework of projective space. 

\medbreak 

Starting afresh, let $\pi \in \PPi$ be a plane on which $P_0, P_1, P_2, P_3 \in \pi^{\h}$ are points. We shall call $P_0 P_1 P_2 P_3$ a {\it complete quadrangle} iff no three of the four points are collinear: that is, iff $\{ P_i, P_j, P_k \}^{\h} = \{ \pi \}$ whenever $i, j, k \in \{ 0, 1, 2, 3 \}$ are distinct. The {\it diagonal points} of the complete quadrangle $P_0 P_1 P_2 P_3$ are $D_1, D_2, D_3 \in \pi^{\h}$ defined by 
$$\{ D_1 \} = \{ P_0, P_1 \}^{\h \h} \cap \{ P_2, P_3 \}^{\h \h},$$
$$\{ D_2 \} = \{ P_0, P_2 \}^{\h \h} \cap \{ P_3, P_1 \}^{\h \h},$$
$$\{ D_3 \} = \{ P_0, P_3 \}^{\h \h} \cap \{ P_1, P_2 \}^{\h \h}.$$
Note that these points are well-defined: for example, 
$$(P_1 P_2)_{\PPi} \cap (P_3 P_0)_{\PPi} =  \{ P_1, P_2 \}^{\h} \cap \{ P_3, P_0 \}^{\h} = \{ \pi \}$$
whence Theorem 2 of [4] ensures that 
$$| (P_1 P_2)_{\PP} \cap (P_3 P_0)_{\PP} | =  | \{ P_1, P_2 \}^{\h \h} \cap \{ P_3, P_0 \}^{\h \h} | = 1.$$
Theorem 1 of [4] ensures that requiring $D_1, D_2, D_3$ to be noncollinear is equivalent to requiring $\{ D_1, D_2, D_3 \}^{\h}$ to be a singleton. 

\medbreak 

With this preparation, our version of the harmonicity axiom reads as follows. 

\medbreak 

\noindent 
AXIOM [{\bf H}] \\ 
If $P_0 P_1 P_2 P_3$ is a complete quadrangle in the plane $\pi$ then $\{ D_1, D_2, D_3 \}^{\h} = \{ \pi \}.$

\medbreak 

As presented, this harmonicity axiom is not manifestly self-dual. In order to verify that the crucial principle of duality is preserved when AXIOM [1] - AXIOM [4] are augmented by AXIOM [{\bf H}], we shall demonstrate that AXIOM [{\bf H]} implies its own (spatial) dual. To this end, let $P \in \PP$ be a point and let $\pi_0, \pi_1, \pi_2, \pi_3$ be planes such that $\{ \pi_i, \pi_j, \pi_k \}^{\h} = \{ P \}$ whenever $i, j, k \in \{ 0, 1, 2, 3 \}$ are distinct; the corresponding diagonal planes are (dually) well-defined by 
$$\{ \delta_1 \} = \{ \pi_0, \pi_1 \}^{\h \h} \cap \{ \pi_2, \pi_3 \}^{\h \h},$$
$$\{ \delta_2 \} = \{ \pi_0, \pi_2 \}^{\h \h} \cap \{ \pi_3, \pi_1 \}^{\h \h},$$
$$\{ \delta_3 \} = \{ \pi_0, \pi_3 \}^{\h \h} \cap \{ \pi_1, \pi_2 \}^{\h \h}.$$
Assuming that AXIOM [{\bf H}] holds, our task is to prove that $\{ \delta_1, \delta_2, \delta_3 \}^{\h} = \{ P \}$. 

\medbreak 

We begin our proof by fixing a plane $\omega \notin P^{\h}$ as provided by AXIOM [1]. Our proof now runs through a sequence of claims. Diagrams illustrating the disposition of the various points and planes involved make these claims eminently plausible; we offer a proof of each within our point-plane framework. 

\medbreak 

\noindent
{\bf Claim 1}: If $i, j \in \{ 0, 1, 2, 3 \}$ are distinct then $\{ \omega, \pi_i, \pi_j \}^{\h}$ is a singleton $\{ P_{i j} \}.$ 

\medbreak 

[It cannot be that $\omega,  \pi_i$ and $\pi_j$ are collinear, for that would force $\{ \pi_i, \pi_j \}^{\h} \subseteq \omega^{\h}$, whereas $P \in \{ \pi_i, \pi_j \}^{\h}$ and $P \notin \omega^{\h}$ by choice. Now invoke Theorem 1 of [4].] 

\medbreak 

\noindent
{\bf Claim 2}: If $i, j \in \{ 0, 1, 2, 3 \}$ are distinct then $\{ \omega, \pi_i \}^{\h \h} \cap \{ \omega, \pi_j \}^{\h \h} = \{ \omega \}.$ 

\medbreak 

[The indicated intersection contains $\omega$. Notice that $\{ \omega, \pi_i \}^{\h} \cap \{ \omega, \pi_j \}^{\h} = \{ \omega, \pi_i, \pi_j \}^{\h} = \{ P_{i j} \}$ by Claim 1. Now invoke Theorem 2 of [4].]

\medbreak 

\noindent 
{\bf Claim 3}: If $i, j \in \{ 0, 1, 2, 3 \}$ are distinct then $\{ P, P_{i j} \}^{\h \h} = \{ \pi_i , \pi_j \}^{\h}.$ 

\medbreak 

[$P \ne P_{i j}$ because $P \notin \omega^{\h} \ni P_{i j}$; also, $\pi_i \ne \pi_j$. By hypothesis, $P \in \{ \pi_i, \pi_j \}^{\h}$; by Claim 1, $P_{i j} \in \{ \pi_i, \pi_j \}^{\h}$. Now invoke AXIOM [4].]

\medbreak 

\noindent 
{\bf Claim 4}: If $i, j, k \in \{ 0, 1, 2, 3 \}$ are distinct then $\{ P_{k i}, P_{k j} \}^{\h \h} = \{ \omega, \pi_k \}^{\h}$. 

\medbreak 

[The points $P_{k i}$ and $P_{k j}$ are distinct; each of them lies in both $\omega^{\h}$ and $\pi_k^{\h}$. Now invoke AXIOM [4] again.]

\medbreak 

\noindent
{\bf Claim 5}:  If $\{ i, j, k \} = \{ 1, 2, 3 \}$ then $\{ P, P_{i j} , P_{k 0} \}^{\h} = \{ \delta_k \}.$ 

\medbreak 

[Claim 3 implies that $ \{ \pi_i, \pi_j \}^{\h \h} = \{ P, P_{i j} \}^{\h}$ and $ \{ \pi_k, \pi_0 \}^{\h \h} = \{ P, P_{k 0} \}^{\h}$ from which follows $\{ \delta_k \} = \{ \pi_i, \pi_j \}^{\h \h} \cap \{ \pi_k, \pi_0 \}^{\h \h} = \{ P, P_{i j} \}^{\h} \cap \{ P, P_{k 0} \}^{\h} = \{ P, P_{i j}, P_{k 0} \}^{\h}$.]

\medbreak 

\noindent 
{\bf Claim 6}:  If $\{ i, j, k \} = \{ 1, 2, 3 \}$ then $\{ P_{i j} , P_{k 0} \}^{\h \h} = \{ \omega, \delta_k \}^{\h}.$ 

\medbreak 

[Each of $P_{i j}$ and $P_{k 0}$ lies in $\omega^{\h}$ by construction; each lies in $\delta_k^{\h}$ on account of Claim 5. Now invoke AXIOM [4] yet again.]

\medbreak 

\noindent
{\bf Claim 7}: If $\{ i, j, k \} = \{ 1, 2, 3 \}$ then $\{ P_{0 i} , P_{j k} \}^{\h \h} \cap \{ P_{0 j} , P_{k i} \}^{\h \h}$ is a singleton $\{ D_{i j} \}$. 

\medbreak 

[$\{ P_{0 i} , P_{j k} \}^{\h} \cap \{ P_{0 j} , P_{k i} \}^{\h} = \{ P_{0 i}, P_{0 j} \}^{\h} \cap \{P_{j k} , P_{k i} \}^{\h} = \{ \omega, \pi_0 \}^{\h \h} \cap \{ \omega, \pi_k \}^{\h \h} = \{ \omega \}$ by Claim 4 and Claim 2. Now Theorem 2 of [4] applies.] 

\medbreak 

\noindent 
{\bf Claim 8}: $P_{0 1} P_{0 2} P_{2 3} P_{3 1}$ is a complete quadrangle in the plane $\omega.$ 

\medbreak 

[Notice that $\{ P_{0 1}, P_{0 2} \}^{\h \h} \cap \{ P_{0 2}, P_{2 3} \}^{\h \h} = \{\omega, \pi_0, \pi_2 \}^{\h} = \{ P_{0 2} \}$ by Claim 4 and Claim 1. Now Theorem 2 of [4] shows that $\{ P_{0 1}, P_{0 2}, P_{2 3} \}^{\h} = \{ P_{0 1}, P_{0 2} \}^{\h} \cap \{ P_{0 2}, P_{2 3} \}^{\h}$ is a singleton. The other three triples succumb to similar arguments.]

\medbreak 

\noindent
{\bf Claim 9}: $\{ \delta_1, \delta_2, \delta_3 \}^{\h} = \{ P \}.$ 

\medbreak 

[$\{ \omega, \delta_1 \}^{\h} \cap \{ \omega, \delta_2 \}^{\h} \cap \{ \omega, \delta_3 \}^{\h} = \{ P_{0 1}, P_{2 3} \}^{\h \h} \cap \{ P_{0 2}, P_{3 1} \}^{\h \h} \cap \{ P_{0 3}, P_{1 2} \}^{\h \h}$ by Claim 6 so that $\omega^{\h} \cap \{ \delta_1, \delta_2, \delta_3 \}^{\h} = \{ D_{1 2} \} \cap  \{ P_{0 3}, P_{1 2} \}^{\h \h}$ by Claim 7. The complete quadrangle $P_{0 1} P_{0 2} P_{2 3} P_{3 1}$ of Claim 8 has diagonal points $D_{1 2}$, $ P_{1 2}$ and $P_{3 0}$. According to AXIOM [{\bf H}] these diagonal points are not collinear: thus $D_{1 2} \notin \{ P_{1 2}, P_{3 0} \}^{\h \h}$ and so $\omega^{\h} \cap \{ \delta_1, \delta_2, \delta_3 \}^{\h} = \emptyset$. If $\{ \delta_1, \delta_2, \delta_3 \}^{\h}$ were to contain more than one point, then it would follow that $\{ \delta_1, \delta_2, \delta_3 \}^{\h} = \{\delta_1, \delta_2 \}^{\h}$ and therefore that $\omega^{\h} \cap \{ \delta_1, \delta_2, \delta_3 \}^{\h} = \{ \omega, \delta_1, \delta_2 \}^{\h} \ne \emptyset$ by virtue of AXIOM [3]. Thus $ \{ \delta_1, \delta_2, \delta_3 \}^{\h}$ contains precisely one point, namely $P$.] 

\medbreak 

Our task is done. 

\medbreak 

\begin{theorem} 
{\rm AXIOM} {\rm [{\bf H}]} is equivalent to its dual. 
\end{theorem} 

\begin{proof} 
The claims leading up to the theorem establish that within our self-dual point-plane framework, AXIOM [{\bf H}] implies its own dual. By the principle of duality in our self-dual framework, it follows that the dual of [{\bf H}] implies [{\bf H}] itself. 
\end{proof} 

\medbreak 

The principle of duality here is spatial. Of course, AXIOM [{\bf H}] also has a planar dual: it states that the diagonal lines of a complete quadrilateral are nonconcurrent. 

\bigbreak 

\section*{Projectivity}

\medbreak 

The classical Veblen-Young formulation of projective geometry admits many equivalent versions of the axiom of {\it projectivity}. As the provisional `assumption of projectivity' in Section 35 of [6] it reads thus: `If a projectivity leaves each of three distinct points of a line invariant, it leaves every point of the line invariant'; this leads immediately to the fundamental theorem of projective geometry. In [3] it was natural to assume a version of the axiom that made reference only to lines and their abstract incidence; such a version was extracted from Section 103 of [6] and has the virtue of being manifestly self-dual. Here, we choose to employ another equivalent version,  namely the {\it Pappus} `theorem'; we do so partly in order to illustrate the process of translation between the traditional framework and the point-plane framework. 

\medbreak 

We prepare the way with some definitions; for variety, we use some traditional notation and terminology. Let the distinct lines $\ell_1$ and $\ell_2$ be incident, with $\ell_1 \cdot \ell_2 = O$ and $\ell_1 \square \ell_2 = \omega$. A simple hexagon alternately inscribed in the line-pair $(\ell_1, \ell_2)$ has vertices $A_1 B_2 C_1 A_2 B_1 C_2$ distinct from $O$ and from each other, with $A_1, B_1, C_1 \in \ell_1$ and $A_2, B_2, C_2 \in \ell_2$. The six sides of the hexagon fall into three pairs of opposites, thus: $B_1 C_2, B_2 C_1; \; C_1 A_2, C_2 A_1; \; A_1 B_2, A_2 B_1$. The classical Pappus `theorem' asserts the collinearity of the {\it cross-joins} 
$$A_0 = (B_1 C_2) \cdot ( B_2 C_1), \; B_0 = (C_1 A_2) \cdot ( C_2 A_1), \; C_0 = (A_1 B_2) \cdot (A_2 B_1).$$

\medbreak 

With this understanding, we state our version of the projectivity axiom as follows

\medbreak 

\noindent 
AXIOM [{\bf P}] \\ 
If a simple hexagon is alternately inscribed in an incident line-pair, then its pairs of opposite sides meet in collinear points. 

\medbreak 

As was the case for harmonicity, this axiom is not manifestly self-dual. Accordingly, we shall demonstrate that AXIOM [{\bf P}] implies its own (spatial) dual within our point-plane framework. Again let $(\ell_1, \ell_2)$ be an incident line-pair with  $\ell_1 \cdot \ell_2 = O$ and $\ell_1 \square \ell_2 = \omega$. Let $\alpha_1, \beta_1, \gamma_1, \alpha_2, \beta_2, \gamma_2$ be planes distinct from $\omega$ and from each other, with $\alpha_1, \beta_1, \gamma_1 \in \ell_1$ and $\alpha_2, \beta_2, \gamma_2 \in \ell_2$. Our task is to show collinearity of the planes 
$$\alpha_0 = (\beta_1 \gamma_2) \square (\beta_2 \gamma_1), \; \beta_0 = (\gamma_1 \alpha_2) \square (\gamma_2 \alpha_1), \; \gamma_0 = (\alpha_1 \beta_2) \square (\alpha_2 \beta_1).$$

\medbreak 

Our approach to this task will differ from the approach we took to the corresponding task in the previous section. Rather than proceed in detail strictly within our point-plane framework, we shall freely incorporate guiding comments from the traditional framework.

\medbreak 

Choose and fix a plane $\pi \notin O^{\h}$ by which to section the entire figure. In traditional terms, the trace of the line-pair  $(\ell_1, \ell_2)$ is a point-pair $(P_1, P_2)$ while the traces of the axial pencils of planes $\alpha_1, \beta_1, \gamma_1$ (with axis $\ell_1$) and $\alpha_2, \beta_2, \gamma_2$ (with axis $\ell_2$) are flat pencils of lines $a_1, b_1, c_1$ (with centre $P_1$) and $a_2, b_2, c_2$ (with centre $P_2$). In point-plane terms, note that if $k \in \{ 1, 2 \}$ then by collinearity
$$(\ell_k)_{\PP} = \{ \alpha_k, \beta_k, \gamma_k \}^{\h} = \{ \beta_k, \gamma_k \}^{\h} = \{ \gamma_k, \alpha_k \}^{\h} = \{ \alpha_k, \beta_k \}^{\h}$$
so that 
$$\{ P_k \} = \{ \pi, \beta_k, \gamma_k \}^{\h} = \{ \pi, \gamma_k, \alpha_k \}^{\h} = \{ \pi, \alpha_k, \beta_k \}^{\h}$$
while 
$$(a_k)_{\PP} = \{ \pi, \alpha_k \}^{\h}, \; (b_k)_{\PP} = \{ \pi, \beta_k \}^{\h}, \; (c_k)_{\PP} = \{ \pi, \gamma_k \}^{\h}.$$

\medbreak 

Notice that if $\pi_1 \in \ell_1$ and $\pi_2 \in \ell_2$ are distinct then $\{ \pi, \pi_1, \pi_2 \}^{\h}$ is a singleton: otherwise, $\pi, \pi_1$ and $\pi_2$ would be collinear so that $\{ \pi_1, \pi_2 \}^{\h} \subseteq \pi^{\h}$ whereas $O \in \{ \pi_1, \pi_2 \}^{\h}$ but $O \notin \pi^{\h}$. In particular, if $\{ i, j \} = \{ 1, 2 \}$ then $\{ \pi, \alpha_i, \beta_j \}^{\h}$ is a singleton: in fact, 
$$\{ \pi, \alpha_i, \beta_j \}^{\h} = \{ \pi, \alpha_i \}^{\h} \cap \{ \pi, \beta_j \}^{\h} = (a_i)_{\PP} \cap (b_j)_{\PP} = \{ a_i \cdot b_j \}.$$
We {\bf claim} further that the line joining $a_1 \cdot b_2$ and $a_2 \cdot b_1$ is given by 
$$((a_1 \cdot b_2) (a_2 \cdot b_1))_{\PP} = \{ a_1 \cdot b_2, a_2 \cdot b_1 \}^{\h \h} = \{ \pi, \gamma_0 \}^{\h}.$$
To see this, note that $\{ a_1 \cdot b_2, a_2 \cdot b_1 \} \h \{ \pi, \gamma_0 \}$: on the one hand, $a_i \cdot b_j \in \{ \pi, \alpha_i, \beta_j \}^{\h} \subseteq \pi^{\h}$; on the other hand, $\gamma_0 = (\alpha_1 \beta_2) \square (\alpha_2 \beta_1)$ so that $\gamma_0  \in (\alpha_i \beta_j)_{\PPi} = \{ \alpha_i, \beta_j \}^{\h \h}$ and therefore $a_i \cdot b_j \in \{ \pi, \alpha_i, \beta_j \}^{\h} \subseteq \{ \alpha_i, \beta_j \}^{\h} \subseteq \gamma_0^{\h}$. Our {\bf claim} now follows by AXIOM [4]. Of course, corresponding assertions apply to the lines $(b_1 \cdot c_2) (b_2 \cdot c_1)$ and $(c_1 \cdot a_2) (c_2 \cdot a_1)$. 

\medbreak 

We now follow the familiar approach to realizing that the planar dual of [{\bf P}] is essentially [{\bf P}] itself. Thus, we consider the pencils of points $P_1, b_2 \cdot c_1, c_1 \cdot a_2$ (on $c_1$) and $P_2, c_2 \cdot a_1, b_1 \cdot c_2$ (on $c_2$). AXIOM [{\bf P}] declares the collinearity of the cross-joins 
$$a_1 \cdot b_2, \; a_2 \cdot b_1, \; ((b_1 \cdot c_2) (b_2 \cdot c_1)) \cdot ((c_1 \cdot a_2) (c_2 \cdot a_1))$$
whence the point $S =  ((b_1 \cdot c_2) (b_2 \cdot c_1)) \cdot ((c_1 \cdot a_2) (c_2 \cdot a_1)) \ne O$ lies on each of the lines 
$$(b_1 \cdot c_2) (b_2 \cdot c_1), \; (c_1 \cdot a_2) (c_2 \cdot a_1), \; (a_1 \cdot b_2) (a_2 \cdot b_1)$$
and in each of the sets 
$$\{ \pi, \alpha_0 \}^{\h}, \; \{ \pi, \beta_0 \}^{\h}, \; \{ \pi, \gamma_0 \}^{\h}.$$
Finally, $\gamma_0 \in O^{\h}$ and $\gamma_0 \in S^{\h}$ (because $S \in \{ \pi, \gamma_0 \}^{\h} \subseteq \gamma_0^{\h}$) so that $\gamma_0 \in \{ O, S \}^{\h}$ while $\{ O, S \}^{\h}$ contains $\alpha_0$ and $\beta_0$ likewise: that is, the planes $\alpha_0, \beta_0, \gamma_0$ are collinear. 

\medbreak 

Our task is done. 

\medbreak 

\begin{theorem} 
{\rm AXIOM} {\rm [{\bf P}]} is equivalent to its dual. 
\end{theorem} 

\begin{proof} 
The argument leading up to the theorem shows that within our point-plane framework, AXIOM [{\bf P}] implies its own dual. By the principle of duality in our framework, it follows that the dual of [{\bf P}] implies [{\bf P}] itself. 
\end{proof} 

\medbreak 

Again, the dual here is spatial. AXIOM [{\bf P}] also has a planar dual: as noted in the course of our argument, this planar dual amounts to [{\bf P}] itself. 

\medbreak 

We close by remarking that, if we were only interested in presenting a suitable version of the projectivity axiom,  it would be arguably better first to define a projectivity of a one-dimensional primitive form and then to adopt the following version: `if a projectivity fixes three of the lines in a flat pencil then it fixes all'. Symbolically: if $O \in \PP$ and $\omega \in \PPi$  are incident, then the flat pencil comprising all lines through $O$ and on $\omega$ is $\Lambda _{O, \omega} = \{ \ell : (O \in \ell_{\PP}) \wedge (\omega \in \ell_{\PPi}) \}$; our version of the projectivity axiom would then state that the only projectivity fixing three elements of $\Lambda_{ O, \omega }$ is the identity. This version is aesthetically superior in being manifestly self-dual. 

\bigbreak 

\begin{center} 
{\small R}{\footnotesize EFERENCES}
\end{center} 
\medbreak 

[1] H.S.M. Coxeter, {\it Projective Geometry}, Second Edition, Springer-Verlag (1987). 

\medbreak 

[2] P.L. Robinson, {\it Projective Space: Lines and Duality}, arXiv 1506.06051 (2015). 

\medbreak 

[3] P.L. Robinson, {\it Projective Space: Reguli and Projectivity}, arXiv 1506.08217 (2015). 

\medbreak 

[4] P.L. Robinson, {\it Projective Space: Points and Planes}, arXiv 1611.06852 (2016). 

\medbreak 

[5] P.L. Robinson, {\it Projective Space: Tetrads and Harmonicity}, arXiv 1612.01913 (2016). 

\medbreak 

[6] O. Veblen and J.W. Young, {\it Projective Geometry}, Volume I, Ginn and Company (1910).

\medbreak

\end{document}